\documentclass[10pt]{amsart}
\usepackage{epsfig}
\usepackage{graphicx}
\usepackage{amscd}
\usepackage{amsmath}
\usepackage{amsxtra}
\usepackage{amsfonts}
\usepackage{amssymb}

\newtheorem{theorem}{Theorem}[section]
\newtheorem{corollary}[theorem]{Corollary}
\newtheorem{lemma}[theorem]{Lemma}
\newtheorem{proposition}[theorem]{Proposition}

\theoremstyle{definition}
\newtheorem{definition}[theorem]{Definition}
\newtheorem{remark}[theorem]{Remark}

\newtheorem{example}[theorem]{Example}
\theoremstyle{remark}

\renewcommand{\theclaim}{\textup{\theclaim}}

\newtheorem*{acknowledgements}{Acknowledgements}

\numberwithin{equation}{section}

\newcommand\restr[2]{{
  \left.\kern-\nulldelimiterspace 
  #1 
  \vphantom{\big|} 
  \right|_{#2} 
  }}

\def\openone

{\mathchoice

{\hbox{\upshape \small1\kern-3.3pt\normalsize1}}

{\hbox{\upshape \small1\kern-3.3pt\normalsize1}}

{\hbox{\upshape \tiny1\kern-2.3pt\SMALL1}}

{\hbox{\upshape \Tiny1\kern-2pt\tiny1}}}

\makeatletter

\newbox\ipbox

\newcommand{\ip}[2]{\left\langle #1\, , \,#2\right\rangle}
\newcommand{\diracb}[1]{\left\langle #1\mathrel{\mathchoice

{\setbox\ipbox=\hbox{$\displaystyle \left\langle\mathstrut
#1\right.$}

\vrule height\ht\ipbox width0.25pt depth\dp\ipbox}

{\setbox\ipbox=\hbox{$\textstyle \left\langle\mathstrut
#1\right.$}

\vrule height\ht\ipbox width0.25pt depth\dp\ipbox}

{\setbox\ipbox=\hbox{$\scriptstyle \left\langle\mathstrut
#1\right.$}

\vrule height\ht\ipbox width0.25pt depth\dp\ipbox}

{\setbox\ipbox=\hbox{$\scriptscriptstyle \left\langle\mathstrut
#1\right.$}

\vrule height\ht\ipbox width0.25pt depth\dp\ipbox}

}\right. }

\newcommand{\dirack}[1]{\left. \mathrel{\mathchoice

{\setbox\ipbox=\hbox{$\displaystyle \left.\mathstrut
#1\right\rangle$}

\vrule height\ht\ipbox width0.25pt depth\dp\ipbox}

{\setbox\ipbox=\hbox{$\textstyle \left.\mathstrut
#1\right\rangle$}

\vrule height\ht\ipbox width0.25pt depth\dp\ipbox}

{\setbox\ipbox=\hbox{$\scriptstyle \left.\mathstrut
#1\right\rangle$}

\vrule height\ht\ipbox width0.25pt depth\dp\ipbox}

{\setbox\ipbox=\hbox{$\scriptscriptstyle \left.\mathstrut
#1\right\rangle$}

\vrule height\ht\ipbox width0.25pt depth\dp\ipbox}

} #1\right\rangle}

\newcommand{\bd}{\mathbb D}

\newcommand{\Sol}{\operatorname*{Sol}}

\newcommand{\cj}[1]{\overline{#1}}

\newcommand{\bz}{\mathbb{Z}}
\newcommand{\M}{\mathcal{M}}
\newcommand{\B}{\mathcal{B}}
\newcommand{\br}{\mathbb{R}}
\newcommand{\bc}{\mathbb{C}}
\newcommand{\bt}{\mathbb{T}}
\newcommand{\bn}{\mathbb{N}}

\def\blfootnote{\xdef\@thefnmark{}\@footnotetext}

\renewcommand{\mod}{\operatorname{mod}}

\input xy
\xyoption{all}
\usepackage{amssymb}

\newcommand{\Span}{\overline{\mbox{span}}}

\def\be{\mathbb{E}}

\def\A{\mathcal{A}}

\def\H{\mathcal{H}}

\def\-{^{-1}}
\def\B{\mathcal{B}}

\def\U{\mathcal{U}}

\def\prob{\operatorname*{Prob}}

\def\bp{\mathbb{P}}

\def\C{\mathcal{C}}

\def\A{\mathcal{A}}

\begin{document}

\title[Transfer operators]{The role of transfer operators and shifts in the study of fractals: encoding-models, analysis and geometry, commutative and non-commutative}
\author{Dorin Ervin Dutkay}
\blfootnote{}
\address{[Dorin Ervin Dutkay] University of Central Florida\\
    Department of Mathematics\\
    4000 Central Florida Blvd.\\
    P.O. Box 161364\\
    Orlando, FL 32816-1364\\
U.S.A.\\} \email{Dorin.Dutkay@ucf.edu}

\author{Palle E.T. Jorgensen}
\address{[Palle E.T. Jorgensen]University of Iowa\\
Department of Mathematics\\
14 MacLean Hall\\
Iowa City, IA 52242-1419\\}\email{jorgen@math.uiowa.edu}

\thanks{}

\begin{abstract}
 We study a class of dynamical systems in $L^2$ spaces of infinite products $X$. Fix a compact Hausdorff space $B$. Our setting encompasses such cases when the dynamics on $X = B^\bn$ is determined by the one-sided shift in $X$, and by a given transition-operator $R$. Our results apply to any positive operator $R$ in $C(B)$ such that $R1 = 1$. From this we obtain induced measures  $\Sigma$ on $X$, and we study spectral theory in  the associated $L^2(X,\Sigma)$.

For the second class of dynamics, we introduce a fixed endomorphism $r$ in the base space $B$, and specialize to the induced solenoid $\Sol(r)$. The solenoid $\Sol(r)$ is then naturally embedded in $X = B^\bn$, and $r$ induces an automorphism in $\Sol(r)$. The induced systems will then live in  $L^2(\Sol(r), \Sigma)$.

The applications include wavelet analysis, both in the classical setting of $\br^n$, and Cantor-wavelets in the setting of fractals induced by affine iterated function systems (IFS). But our solenoid analysis includes such hyperbolic systems as the Smale-Williams attractor, with the endomorphism  $r$ there prescribed to preserve a foliation by meridional disks. And our setting includes the study of Julia set-attractors in complex dynamics.
\end{abstract}
\maketitle \tableofcontents
\section{Introduction}

The purpose of this paper is to offer a general framework for geometry and analysis of iteration systems. We offer a setting encompassing the kind of infinite product, or solenoid constructions arising in the study of iterated function systems (IFSs). Our aim is to give an operator theoretic construction of infinite product measures in a general setting that includes wavelet analysis of IFSs. To motivate this, recall, that to every affine function system  $S$  with fixed scaling matrix, and a fixed set of translation points in $\br^n$, we may associate to $S$ a solenoid. By this we mean a measure space whose $L^2$ space includes $L^2(\br^n)$ in such a way that $\br^n$ embeds densely in the solenoid. (In the more familiar case of $n = 1$, we speak of a dense curve in an infinite-dimensional ``torus''. The latter being a geometric model of the solenoid.

   The need for this generality arose in our earlier investigations, for example in the building of wavelet systems on Cantor systems, of which the affine IFSs are special cases. In these cases (see Theorem \ref{th1.6} and Corollary \ref{cor4.5} below) we found that one must pass to a suitable $L^2$ space of a solenoid. Indeed, we showed that such wavelet bases fail to exist in the usual receptor Hilbert space $L^2(\br^n)$ from wavelet theory.

    For reference to earlier papers dealing with measures on infinite products, and their use in harmonic analysis and wavelet theory on fractals; see e.g., \cite{DJ12, DJS12, DLS11, DJ11a, DS11, DJ11b, DHS11, DJ10, DJP09,LN12, DL10, LW09}.

The paper is organized as follows: Starting with a compact Hausdorff space $B$, and a positive operator $R$  in $C(B)$, we pass to a family of induced probability measures $\Sigma$ (depending on $R$) on the infinite product $\Omega = B^\bn$.  Among all probability measures on $\Omega$, we characterize those which are induced.
In section 2, we prove a number of theorems about $R$-induced measures on $\Omega$, and we include applications to random walks, and to fractal analysis.
 In section 3, we then introduce an additional structure: a prescribed endomorphism  $r$ in the base space $B$, and we study the corresponding solenoid  $\Sol(r)$, contained in $\Omega$, and its harmonic analysis, including applications to generalized wavelets. The latter are studied in detail in section 4 where we introduce wavelet-filters, in the form of certain functions $m$ on $B$.

\section{Analysis of infinite products}

\begin{definition}\label{def1.1}
Let $B$ be some compact Hausdorff space. $\B$ refes to a $\sigma$-algebra, usually generated by the open sets, so Borel. We will denote by $\C$ the cylinder sets, see below. 
We denote by $\M(B)$ the set of positive Borel masures on $B$, and by $\M_1(B)$ those that have $\mu(B)=1$. 
Let $V$ be some set. 
$$B^V=\prod_V B=\mbox{ all functions from $V$ to $B$}.$$
For example $V=\bn$ or $\bz$. 

For $x\in B^V$ we denote by $\pi_v(x):=x_v$, $v\in V$. 
If $V=\bn$, then we denote by 
$$\pi_1^{-1}(x)=\{(x_1,x_2,\dots)\in B^\bn : x_1=x\}.$$

Let $r:B\rightarrow B$ be some onto mapping, and $\mu$ a Borel probability measure on $B$, $\mu(B)=1$.

To begin with we do not introduce $\mu$ and $r$, but if $r$ is fixed and 
\begin{equation}
1\leq \#r^{-1}(x)<\infty,\mbox{ for all }x\in B,
\label{eq1.1.1}
\end{equation}
then we introduce two objects 
\begin{enumerate}
	\item[(1)] $R=R_W$, the Ruelle operator;
	\item[(2)] $\Sol(r)$, the solenoid.
\end{enumerate}
For (1), fix $W:B\rightarrow[0,\infty)$ such that 
$$\sum_{r(y)=x}W(y)=1,\mbox{ for all }x\in B,$$
and set 
\begin{equation}
(R_W\varphi)(x)=\sum_{r(y)=x}W(y)\varphi(y).
\label{eq1.3}
\end{equation}

For (2),
\begin{equation}
\Sol(r)=\left\{x\in B^{\bn} : r(x_{i+1})=x_i, i=1,2,\dots\right\}
\label{eq1.4}
\end{equation}
\begin{equation}
\sigma(x)_i=x_{i+1},\quad (x\in B^{\bn}),\quad \quad \widehat r(x)=(r(x_1),x_1,x_2,\dots). 
\label{eq1.5}
\end{equation}

More generally, consider
\begin{equation}
R:C(B)\rightarrow C(B)\mbox{ or }R:M(B)\rightarrow M(B),
\label{eq1.6}
\end{equation}
where $M(B)$ is the set of all measurable functions on $B$.

\end{definition}

\begin{definition}\label{def1.2}
We say that $R$ is {\it positive} iff 
\begin{equation}
\varphi(x)\geq 0\mbox{ for all }x\in B\mbox{ implies $(R\varphi)(x)\geq 0$, for all $x\in B$}.
\label{eq1.7}
\end{equation}
We will always assume $R1=1$  where $1$ indicates the constant function 1 on $B$. This is satisfied if $R=R_W$ in \eqref{eq1.3}, but there are many other positive operators $R$ with these properties.
\end{definition}

 While what we call ``the transfer operator'' or a ``Ruelle operator'' has a host of distinct mathematical incarnations, each dictated by a particular family of applications, they are all examples of positive operators $R$ in the sense of our Definition \ref{def1.2}. Our paper has two aims: One is to unify, and extend earlier studies; and the other is to prove a number of theorems on measures, dynamical systems, stochastic processes built from infinite products. Indeed there are many positive operators $R$ which might not fall in the class of operators studied as ``transfer operators''. The earlier literature on transfer operators includes applications to physics \cite{LR69}, to the Selberg zeta function \cite{FM12}, to dynamical zeta functions \cite{Ru02, Ru96, Na12, MMS12}; to $C^*$-dynamical systems \cite{Kw12, ABL11}; to the study of Hausdorff dimension  \cite{He12}; to spectral theory \cite{ABL12}.

        These applications are, in addition to the aforementioned, to analysis on fractals, and to generalized wavelets. For book treatments, we refer the reader to \cite{Ba00}, and \cite{BJ02}. The literature on positive operators $R$, in the general sense, is much less extensive; but see \cite{Ar86}.

\begin{definition}\label{def1.3}
A subset $S$ of $B^{\bn}$ is said to be {\it shift-invariant} iff $\sigma(S)\subset S$, where $\sigma$ is as in \eqref{eq1.5}, $\sigma(x)_i=x_{i+1}$. 

\end{definition}

\begin{remark} Every solenoid $\Sol(r)$ is shift-invariant. 
\end{remark}

\begin{example}\label{ex1.4.1}
The solenoids introduced in connection with generalized wavelet constructions:

Let $r:B\rightarrow B$ as above and let $\mu$ be a strongly invariant measure, i.e.,
$$\int f\,d\mu=\int\frac{1}{\#r^{-1}(x)}\sum_{r(y)=x}f(y)\,d\mu(x)$$
for all $f\in C(B)$. 

A {\it quadrature mirror filter (QMF)} for $r$ is a function $m_0$ in $L^\infty(B,\mu)$ with the property that 
\begin{equation}
\frac{1}{N}\sum_{r(w)=z}|m_0(w)|^2=1,\quad(z\in B)
\label{eq1.5.1}
\end{equation}

As shown by Dutkay and Jorgensen \cite{DuJo05,DuJo07}, every quadrature mirror filter (QMF) gives rise to a wavelet theory. Various extra conditions on the filter $m_0$ will produce wavelets in $L^2(\br)$ \cite{Dau92}, on Cantor sets \cite{DuJo06w,MaPa11}, on Sierpinski gaskets \cite{Dan08} and many others.
\begin{theorem}\cite{DuJo05,DuJo07}\label{th1.6}
Let $m_0$ be a QMF for $r$. Then there exists a Hilbert space $\H$, a representation $\pi$ of $L^\infty(B)$ on $\H$, a unitary operator $U$ on $\H$ and a vector $\varphi$ in $\H$ such that 
\begin{enumerate}
	\item {\bf (Covariance)} 
	\begin{equation}
U\pi(f)U^*=\pi(f\circ r),\quad(f\in L^\infty(B))
\label{eqw1}
\end{equation}
\item {\bf (Scaling equation)} 
\begin{equation}
U\varphi=\pi(m_0)\varphi
\label{eqw2}
\end{equation}
\item {\bf (Orthogonality)}
\begin{equation}
\ip{\pi(f)\varphi}{\varphi}=\int f\,d\mu,\quad(f\in L^\infty(B))
\label{eqw3}
\end{equation}
\item {\bf (Density)}
\begin{equation}
\Span\left\{U^{-n}\pi(f)\varphi : f\in L^\infty(B),n\geq 0\right\}=\H
\label{eqw4}
\end{equation}
\end{enumerate}

\end{theorem}

The system $(\H,U,\pi,\varphi)$ in Theorem \ref{th1.6} is called {\it the wavelet representation associated to the QMF $m_0$}.

While, as we mentioned before, these representations can have incarnations on the real line, or on Cantor sets, they can be also represented using certain random-walk measures on the solenoid (see \cite{DuJo05,DuJo07,Dut06}).
\end{example}

\begin{remark}\label{rem1.4.1}
In examples when the condition \eqref{eq1.1.1} is not satisfied, the modification of the family of relevant integral operators is as follows. 

In the general case when $r:B\rightarrow B$ is given, but $\#r^{-1}(x)=\infty$, the modification of the operators $R$, extending those from Example \ref{ex1.4.1}, is as follows:

Consider
\begin{enumerate}
	\item $W:B\rightarrow[0,\infty)$ Borel
	\item $p:B\times\B(B)\rightarrow[0,\infty)$ such that for all $x\in B$, $p(x,\cdot)\in\M(r^{-1}(x))$, so is a positive measure such that 
	$$\int_{r^{-1}(x)}W(y)p(x,dy)=1,\quad(x\in B).$$
\end{enumerate}
Then set 
$$(R\varphi)(x)=\int_{r^{-1}(x)}\varphi(y)W(y)p(x,dy).$$

\end{remark}

\begin{example}\label{ex1.4.2}
$G=(V,E)$ infinite graph, $V$ are the vertices, $E$ are the edges. 
\begin{equation}
i(e)=\mbox{ initial vertex},\quad t(e)=\mbox{terminal vertex}.
\label{eq1.9}
\end{equation}
\begin{equation}
S^{(G)}=\mbox{ solenoid of $G$}=\left\{\tilde e\in E^{\bn} : t(e_j)=i(e_{j+1})\mbox{ for all }j\in\bn\right\}.
\label{eq1.10}
\end{equation}
For example $V=\bz^2$ and  the edges are given by $x\sim y$ iff $\|x-y\|=1$. For details and applications, see \cite{JP10}. 
\end{example}

\begin{definition}\label{def1.5}
Now back to $\C$, the cylinder sets mentioned in Definition \ref{def1.1}. $C\in\C$ are subsets of $B^{V}$ indexed by finite sytems $v_1,\dots, v_n$ , $O_1,\dots,O_n$, $v_i\in V$, $O_i\subset B$ open subsets, $i=1,\dots,n$, $n\in\bn$. 
\begin{equation}
C_{v_i,O_i}:=\left\{\tilde x\in B^V : x_{v_i}\in O_i\mbox{ for all }i=1,\dots,n\right\}.
\label{eq1.12}
\end{equation}

Notation: $\C$ generates the topology and the $\sigma$-algebra of subsets in $B^V$ in the usual way, and $B^V$ is compact by Tychonoff's theorem. 

If $\varphi$ is a function on $B$, we denote by $M_\varphi$ the multiplication operator $M_\varphi f=\varphi f$, defined on functions $f$ on $B$. In the applications below, we will use $C(B)$, all continuous functions from $B$ to $\br$. 

\end{definition}

\begin{lemma}\label{lem1.6}
Consider $B^{\bn}$ and the algebra generated by cylinder functions of the form $f=\varphi_1\otimes\dots\otimes\varphi_n$, $\varphi_i\in C(B)$, $n\in\bn$, $1\leq i\leq n$,
\begin{equation}
(\varphi_1\otimes\dots\otimes\varphi_n)(\tilde x)=\varphi_1(x_1)\varphi_2(x_2)\dots\varphi_n(x_n),\quad (\tilde x\in B^{\bn}),
\label{eq1.13}
\end{equation}
or 
\begin{equation}
\varphi_1\otimes\dots\otimes\varphi_n=(\varphi_1\circ\pi_1)(\varphi_2\circ\pi_2)\dots(\varphi_n\circ\pi_n).
\label{eq1.14}
\end{equation}
Let $\A_\C$ be the algebra of all cylinder functions. Then $\A_\C$ is dense in $C(B^{\bn})$.
\end{lemma}

\begin{proof} Easy consequence of Stone-Weierstrass. 
\end{proof}

\begin{theorem}\label{th1.7}
Let $R$ be a positive operator as in \eqref{eq1.6}, with $R1=1$. Then for each $x\in B$ there exists a unique Borel probability measure $\bp_x$ on $B^{\bn}$ such that 
\begin{equation}
\int_{B^{\bn}}\varphi_1\otimes\dots\otimes \varphi_n\,d\bp_x=\left(M_{\varphi_1}RM_{\varphi_2}\dots RM_{\varphi_n}1\right)(x),\quad(\varphi_i\in C(B), n\in\bn).
\label{eq1.15}
\end{equation}
\end{theorem}

\begin{proof}
We only need to check that the right-hand side of \eqref{eq1.15} for $\varphi_1\otimes \dots\varphi_n$ equals the right-hand side of \eqref{eq1.15} for $\varphi_1\otimes\dots\varphi_n\otimes 1$; but this is immediate from \eqref{eq1.15} and the fact that $R1=1$. The existence and uniqueness of $\bp_x$ the follows form the inductive method of Kolmogorov.
\end{proof}

\begin{corollary}\label{cor1.10}
Let $B$ and $R:C(B)\rightarrow C(B)$ be as in Theorem \ref{th1.7}, and let $\mu\in \M_1(B)$ be given. Let $\Sigma=\Sigma^{(\mu)}$ be the measure on $\Omega=B^{\bn}$ given by
\begin{equation}
\int f\,d\Sigma:=\int_B\int_{\pi_1^{-1}(x)} f\,d\bp_x\,d\mu(x).
\label{eq1.7.1}
\end{equation}
Then 
\begin{enumerate}
	\item $V_1:L^2(B,\mu)\rightarrow L^2(\Omega,\Sigma)$ given by $V_1\varphi:=\varphi\circ\pi_1$ is isometric.
	\item For its adjoint operator $V_1^*$, we have $V_1^*:L^2(\Omega,\Sigma)\rightarrow L^2(B,\mu)$ with 
	\begin{equation}
(V_1^*f)(x)=\int_{\pi_1^{-1}(x)}f\,d\bp_x.
\label{eq1.7.4}
\end{equation}
\end{enumerate}
\end{corollary}

\begin{proof}
The assertion (i) is immediate from Theorem \ref{th1.7}. To prove (ii) we must show that the following formula holds:
\begin{equation}
\int_B\left(\int_{\pi_1^{-1}(x)}f\,d\bp_x\right)\psi(x)\,d\mu(x)=\int_{\Omega}f\,\psi\circ\pi_1\,d\Sigma
\label{eq1.7.5}
\end{equation}
for all $f\in L^2(\Omega,\Sigma)$ and all $\psi\in C(B)$.

Recall that $V_1^*$ is determined by 
\begin{equation}
\ip{V_1^*f}{\psi}_{L^2(\mu)}=\ip{f}{V_1\psi}_{L^2(\Omega,\Sigma)}.
\label{eq1.7.6}
\end{equation}
But by Lemma \ref{lem1.6} (Stone-Weierstrass), to verify \eqref{eq1.7.5}, we may restrict attention to the special case when $f$ has the form given in \eqref{eq1.14}. Note that if $f=(\varphi_1\circ\pi_1)(\varphi_2\circ\pi_2)\dots(\varphi_n\circ\pi_n)$ then 
$$f(\psi\circ\pi_1)=((\varphi_1\psi)\circ\pi_1)(\varphi_2\circ\pi_2)\dots(\varphi_n\circ\pi_n),$$
and so the right-hand side of \eqref{eq1.7.5} is equal to 
$$=\int_\Omega((\varphi_1\psi)\circ\pi_1)(\varphi_2\circ\pi_2)\dots(\varphi_n\circ\pi_n)\,d\Sigma$$
$$=\int_B\varphi_1(x)\psi(x)R(\varphi_2R(\dots\varphi_{n-1}R(\varphi_n))\dots)))(x)\,d\mu(x)=\int_B\psi(x)\int f\,d\bp_x\,d\mu(x)$$
which is the left-hand side of \eqref{eq1.7.5} and (ii) follows.

\end{proof}

\begin{remark}\label{rem2.10}
When $R:C(B)\rightarrow C(B)$ is a given positive operator, we induce measures on $\Omega=B^{\bn}$ by the inductive procedure outlined in the proof of Theorem \ref{th1.7}; but implicit in this construction is an extension of $\varphi\mapsto R(\varphi)$ from all $\varphi$ continuous to all Borel measurable functions. This extension uses the Riesz theorem in the usual way as follows: Fix $x\in B$ and then apply Riesz' theorem to the positive linear functional $C(B)\ni\varphi\mapsto R(\varphi)(x)$. There is a unique regular Borel measure $\mu_x$ on $B$ such that 
$$R(\varphi)(x)=\int_B\varphi(y)\,d\mu_x(y),\quad(\varphi\in C(B)).$$
If $E\subset B$ is Borel, we define $$\tilde R(E)(x)=\tilde R(\chi_E)(x):=\mu_x(E);$$
but we shall use this identification without overly burdening our notation with tildes. 
\end{remark}

\begin{lemma}\label{lemn.1}
Let $B$ and $R$ be specified as above. Given $\mu\in \M_1(B)$, let $\Sigma=\Sigma^{(\mu)}$ denote the corresponding measure on $\Omega$,i.e., 
\begin{equation}
\int_\Omega f\,d\Sigma=\int_B\int_{\pi_1^{-1}(x)}f\,d\bp_x^{(R)}\,d\mu(x)
\label{eqn.3}
\end{equation}
We shall consider $V_1:L^2(B,\mu)\rightarrow L^2(\Omega,\Sigma)$ and its adjoint operator $V_1^*: L^2(\Omega,\Sigma)\rightarrow L^2(B,\mu)$, where $V_1\varphi=\varphi\circ\pi$, for all $\varphi\in L^2(B,\mu)$. Note that the adjoint operator $V_1^*$ makes reference to the choice of $R$ at the very outset. The following two hold:
\begin{enumerate}
	\item $R$ naturally extends to $L^2(B,\mu)$; and
	\item 
	\begin{equation}
RV_1^*f=V_1^*(f\circ\sigma),\quad(f\in L^2(\Omega,\Sigma))
\label{eqn.4}
\end{equation}
\end{enumerate}
\end{lemma}

\begin{remark}\label{remn.2}
Given $R$, we say that a function $\varphi\in B$ is {\it harmonic} iff $R\varphi=\varphi$. It follows that harmonic functions contain the range of $V_1^*$, applied to $\{f: f\circ\sigma=f\}$.
For a stronger conclusion, see Corollary \ref{cor5.3}. 
\end{remark}

\begin{proof}[Proof of Lemma \ref{lemn.1}]
Using the Stone-Weierstrass theorem, applied to $C(\Omega)$, we note that it is enough for us to check the validity of formula \eqref{eqn.4} on the algebra $\A^{(cyl)}$ spanned by all cylinder functions
\begin{equation}
f=(\varphi_1\circ\pi_1)(\varphi_2\circ\pi_2)\dots(\varphi_n\circ\pi_n)
\label{eqn.5}
\end{equation}
$n\in\bn$, $\varphi_i\in C(B)$. But note that if $f$ is as in \eqref{eqn.5} then 
\begin{equation}
f\circ\sigma=(\varphi_1\circ\pi_2)(\varphi_2\circ\pi_3)\dots(\varphi_n\circ\pi_{n+1})
\label{eqn.6}
\end{equation}
Using then \eqref{eq1.7.4} in Corollary \ref{cor1.10} above, we conclude that 
$$(V_1^*(f\circ\sigma))(x)=R(\varphi_1R(\varphi_2(R\dots\varphi_{n-1}R(\varphi_n))\dots))(x)=(RV_1^*f)(x);$$

The extension from the cylinder functions $\A^{(cyl)}$ to all of $L^2(\Omega,\Sigma)$ now follows from the usual application of Stone-Weierstrass; recall that $C(\Omega)$ is dense in $L^2(\Omega,\Sigma)$ relative to the $L^2$-norm;
and we have the desired conclusion. 
\end{proof}

\subsection{What measures on $B^\bn$ have a transfer operator?}

Below we characterize, among all Borel probability measures $\Sigma$  on $B^\bn$,  precisely those which arise  from a pair $\mu$ and $R$ with a transfer operator $R$ and $\mu$ a measure on $B$. The characterization is general and involves only the one-sided shift $\sigma$ on $B^\bn$.

\begin{lemma}\label{lem2.3.1}
Let $\Sigma\in \M_1(B^\bn)$ and set $\mu:=\Sigma\circ\pi_1^{-1}\in\M(B)$; then for $\mu$-almost all $x\in B$ there is a field $\bp_x\in\M(\pi_1^{-1}(x))$ such that 
\begin{equation}
d\Sigma=\int_B\,d\bp_x\,d\mu(x)
\label{eq2.3.1}
\end{equation}
and the following hold
\begin{enumerate}
	\item The operator $V_1:L^2(B,\mu)\rightarrow L^2(B^\bn,\Sigma)$ given by 
	\begin{equation}
V_1\varphi=\varphi\circ\pi_1
\label{eq2.3.2}
\end{equation}
is isometric.
\item Its adjoint operator $V_1^*:L^2(B^\bn,\Sigma)\rightarrow L^2(B,\mu)$ satisfies
\begin{equation}
(V_1^*f)(x)=\int_{\pi_1^{-1}(x)}f\,d\bp_x=:\be_x(f),\quad(x\in B).
\label{eq2.3.3}
\end{equation}
\end{enumerate}
\end{lemma}

\begin{proof}
(i) For $\varphi\in C(B)$, we have 
$$\|V_1\varphi\|_{L^2(\Sigma)}^2=\int_{B^\bn}|\varphi\circ\pi_1|^2\,d\Sigma=\int_{B^\bn}|\varphi|^2\circ\pi_1\,d\Sigma$$
$$=\int_B|\varphi|^2\,d(\Sigma\circ\pi_1^{-1})=\int_B|\varphi|^2\,d\mu.$$

(ii) For $\varphi\in C(B)$ and $f\in L^2(B^\bn,\Sigma)$, we have 
\begin{equation}
\int_{B^\bn}(V_1\varphi)f\,d\Sigma=\int_B\varphi(x)(V_1^*f)(x)\,d\mu(x),
\label{eq2.3.4}
\end{equation}
where $V_1^*f\in L^2(B,\mu)$. Hence 
\begin{equation}
\int_{B^\bn}(\varphi\circ\pi_1)f\,d\Sigma=\int_B\varphi(x)(V_1^*f)(x)\,d\mu(x)
\label{eq2.3.5}
\end{equation}
and $(V_1^*f)(x)$ is well defined for $\mu$-almost all $x\in B$. Moreover, the mapping 
\begin{equation}
C(B^\bn)\ni f\mapsto (V_1^*f)(x)
\label{eq2.3.6}
\end{equation}
is positive; i.e., $f\geq 0$ implies $(V_1^*f)(x)\geq 0$. This follows from \eqref{eq2.3.5}. For if $E\subset B$, $\mu(E)>0$, and $V_1^*f<0$ on $E$ then there exists $\varphi\in C(B)$, $\varphi>0$ such that $\int_B\varphi(x)(V_1^*f)(x)\,d\mu(x)<0$, which contradicts \eqref{eq2.3.5}. Now the conclusion in \eqref{eq2.3.3} follows from an application of Riesz' theorem to \eqref{eq2.3.6}. 

\end{proof}

\begin{proposition}\label{pr2.3.2}
Let $\Sigma\in\M(B^\bn)$, then $(\bp_x)_{x\in B}$ from Lemma \ref{lem2.3.1} has the form \eqref{eq1.15} in Theorem \ref{th1.7} if and only if there is a positive operator $R$ such that $R1=1$ and 
\begin{equation}
\be_x(f\circ\sigma)=(R(\be_\bullet f))(x)
\label{eq2.3.7}
\end{equation}
holds for all $x\in B$ and for all $f\in L^2(B^\bn,\Sigma)$, where in \eqref{eq2.3.7} we use the notation 
\begin{equation}
\be_x(\dots)=\int_{\pi_1^{-1}(x)}\dots\,d\bp_x=\be^{(\Sigma)}(\dots\,|\,\pi_1=x)
\label{eq2.3.8}
\end{equation}
for the field of conditional expectations, and $\be_\bullet f$ denotes the map $x\mapsto \be_xf$.
\end{proposition}

\begin{proof}

The implication \eqref{eq1.15} $\Rightarrow$ \eqref{eq2.3.7} is already established. It is Lemma \ref{lemn.1}(ii). Now assume some positive operator $R$ exists such that \eqref{eq2.3.7} holds. We will then prove that $\Sigma$ is the measure determined in Theorem \ref{th1.7} from $R$ and $\mu=\Sigma\circ\pi_1^{-1}$. It is enough to verify \eqref{eq1.15} on all finite tensors
\begin{equation}
f=(\varphi_1\circ\pi_1)(\varphi_2\circ\pi_2)\dots(\varphi_n\circ\pi_n)
\label{eq2.3.9}
\end{equation}
as in \eqref{eq1.14}; and we now establish \eqref{eq1.15} by induction, using the assumed \eqref{eq2.3.7}. 

The case $n=1$ is 
$$\be_x(\varphi\circ\pi_1)=\varphi(x),\quad(\varphi\in C(B),x\in B);$$
and this follows from Lemma \ref{lem2.3.1}. 

For $n=2$, we compute as follows
\begin{equation}
\be_x(\varphi_1\circ\pi_1\,\varphi_2\circ\pi_2)=\varphi_1(x)(R\varphi_2)(x).
\label{eq2.3.10}
\end{equation}
To do this, we shall prove the following fact, obtained from assumption \eqref{eq2.3.7}:

For $\psi\in C(B)$ and $f\in L^2(B^\bn,\Sigma)$ we have 
\begin{equation}
\be_x((\psi\circ\pi_1)f)=\psi(x)\be_x(f).
\label{eq2.3.11}
\end{equation}
Using \eqref{eq2.3.3} in Lemma \ref{lem2.3.1}(ii), note that \eqref{eq2.3.11} is equivalent to 
$$\int_{B^\bn}(\varphi\circ\pi_1)(\psi\circ\pi_1)f\,d\Sigma=\int_B\varphi\psi V_1^*f\,d\mu,$$
which in turn follows from $V_1(\varphi\psi)=(V_1\varphi)(V_1\psi)$ since $\varphi\mapsto\varphi\circ\pi_1$ is multiplicative.

Returning to \eqref{eq2.3.10}, we then get 
$$\be_x((\varphi_1\circ\pi_1)(\varphi_2\circ\pi_2))=\varphi_1(x)\be_x(\varphi_2\circ\pi_2)=\varphi_1(x)\be_x(\varphi_2\circ\pi_1\circ\sigma)$$
$$=\varphi_1(x)R(\be_\bullet(\varphi_2\circ\pi_1))(x)=\varphi_1(x)(R\varphi_2)(x).$$
We shall now be using $\pi_i\circ\sigma=\pi_{i+1}$. 

Assume that 
\begin{equation}
\be_x(f)=\varphi_1(x)R(\varphi_2 R(\dots R(\varphi_n)\dots))(x)
\label{eq2.3.12}
\end{equation}
holds when $f$ in \eqref{eq2.3.9} has length $n-1$; then we show it must hold if it has length $n$. We set
$$\be_x(f)=\be_x((\varphi_1\circ\pi_1)(g\circ\sigma)),$$
where $g$ is a tensor of length $n-1$. Hence the induction hypothesis yields
$$\be_x(f)=\varphi_1(x)\be_x(g\circ\sigma)=\varphi_1(x)R(\be_\bullet(g))$$
which is the right-hand side of \eqref{eq2.3.12}. 
\end{proof}

\subsection{Subalgebras in $L^\infty(\Omega,\Sigma)$ and a conditional expectation}

Let $B$, $R:C(B)\rightarrow C(B)$, $\mu\in \M_1(B)$ and $\Sigma=\Sigma^{(\mu)}$ be as specified. The only assumptions on $R$ are that 

\begin{enumerate}
\item it is linear;
\item it is positive and 
\item $R1=1$. 
\end{enumerate}

We will be using Theorem \ref{th1.7} and Corollaries \ref{cor1.10} and \ref{cor1.11} referring to  the measures 
\begin{equation}
\{\bp_x^{(R)} : x\in B\}\mbox{ on }\pi_1^{-1}(x),\quad (x\in B).
\label{eq4.1.1}
\end{equation}

The theorem below is about the operators $\{V_n : n\in\bn\}$, $V_n:L^2(B,\mu)\rightarrow L^2(\Omega,\Sigma)$ given by $$V_n\varphi=\varphi\circ\pi_n,\quad(\varphi\in C(B),n\in\bn).$$ Since $V_1: L^2(B,\mu)\rightarrow L^2(\Omega,\Sigma)$ is isometric, it follows that 
\begin{equation}
Q_1:=V_1V_1^*
\label{eq4.1.2}
\end{equation}
is a projection in each of the Hilbert spaces $L^2(\Omega,\Sigma^{(\mu)})$. 

\begin{theorem}\label{th4.1}
With $B,R,\mu, \Sigma=\Sigma^{(\mu)}$ and $V_n$ specified as above, we have the following formulas:
\begin{enumerate}
	\item $V_1^*V_{n+1}=R^n$ on $L^2(B,\mu)$, $n=0,1,2\dots$;
	\item $Q_1:=V_1V_1^*$ is a conditional expectation onto 
	$$\A_1:=\{\varphi\circ\pi_1 :\varphi\in L^\infty(B,\mu)\}$$
\begin{equation}
Q_1((\varphi\circ\pi_1)f)=(\varphi\circ\pi_1)Q_1(f) \mbox{ for all $\varphi\in L^\infty(B,\mu)$, $f\in L^\infty(\Omega,\Sigma)$. }
\label{eq4.1.3}
\end{equation} 
\item $Q_1(\varphi\circ\pi_{n+1})=(R^n\varphi)\circ\pi_1$ for all $\varphi\in C(B)$, $n=0,1,2,\dots$. 
\end{enumerate}
\end{theorem}

\begin{proof} (i) As a special case of Theorem \ref{th1.7}, we see that 
\begin{equation}
\int_{\pi_1^{-1}(x)}(\varphi\circ\pi_{n+1})\,d\bp_x^{(R)}=(R^n\varphi)(x)
\label{eq4.1.4}
\end{equation}
holds for all $\varphi\in C(B)$. We further see that \eqref{eq4.1.4} extends to both $L^\infty(B,\mu)$ and to $L^2(B,\mu)$. Hence
\begin{equation}
(V_1^*V_{n+1}\varphi)(x)=\int_{\pi_1^{-1}(x)}(\varphi\circ\pi_{n+1})\,d\bp_x^{(\br)}=(R^n\varphi)(x),\quad(x\in B).
\label{eq4.1.7}
\end{equation}

(ii) By Lemma \ref{lem1.6}, we see that to verify \eqref{eq4.1.3}, it is enough to check it for cylinder functions $f$, i.e.,
\begin{equation}
f=(\psi_1\circ\pi_1)(\psi_2\circ\pi_2)\dots(\psi_n\circ\pi_n),
\label{eq4.1.8}
\end{equation}
$n\in\bn$, $\psi_i\in C(B)$. But if $f$ is as in \eqref{eq4.1.8}, then 
\begin{equation}
(\varphi\circ\pi_1)f=((\varphi\psi_1)\circ\pi_1)(\psi_2\circ\pi_2)\dots(\psi_n\circ\pi_n),
\label{eq4.1.9}
\end{equation}
and the desired formula \eqref{eq4.1.3} is immediate.

(iii) Given (i), we may apply $V_1$ to both sides in \eqref{eq4.1.7}, and the desired formula (iii) follows. 
\end{proof}

It is important to stress that one obtains a closed-form expression for $V_1^*$ where the operator $V_1:\varphi\mapsto \varphi\circ\pi_1$ is introduced in Corollary \ref{cor1.10}. Indeed $V_1^*$ is a conditional expectation:
\begin{equation}
(V_1^*f)(x)=\be^{(\Sigma)}(f\,|\,\pi_1=x)=\be_x^{(\Sigma)}(f),\quad(x\in B,f\in L^2(\Omega,\Sigma))
\label{eq-1}
\end{equation}
By contrast, the situation for $V_n^*$, $n>1$ is more subtle. 
\begin{proposition}\label{pr-2}
Let $B$, $R:C(B)\rightarrow C(B)$ and $\Sigma=\Sigma^{(\mu)}\in\M_1(\Omega)$ be as above, i.e., $\mu=\Sigma\circ\pi_1^{-1}$. Let $R^*$ be the adjoint of the operator $R$ when considered as a bounded operator in $L^2(B,\mu)$. For $V_2^*$ we have
\begin{equation}
V_2^*((\varphi_1\circ\pi_1)(\varphi_2\circ\pi_2)\dots(\varphi_n\circ\pi_n))=R^*(\varphi_1)\varphi_2R(\varphi_3\dots R(\varphi_n)\dots).
\label{eq-3}
\end{equation}
\end{proposition}

\begin{proof}
Let $\psi$ be the function given on the right hand side in \eqref{eq-3}. The operator $V_2:\varphi\mapsto\varphi\circ\pi_2$ maps from $L^2(B,\mu)$ into $L^2(\Omega,\Sigma)$. The assertion in \eqref{eq-3} follows if we check that, for all $\xi\in C(B)$, we have the following identity:
\begin{equation}
\int_\Omega(\varphi_1\circ\pi_1)((\xi\varphi_2)\circ\pi_2)(\varphi_3\circ\pi_3)\dots(\varphi_n\circ\pi_n)\,d\Sigma=\int_B\xi\psi\,d\mu.
\label{eq-4}
\end{equation}
But we may compare the left-hand side in \eqref{eq-4} with the use of Theorem \ref{th1.7}:
$$=\int_B\varphi_1R((\xi\varphi_2)R(\varphi_3R(\dots R(\varphi_n)\dots)))\,d\mu=\int_B(R^*\varphi_1)\xi\varphi_2R(\varphi_3R(\dots R(\varphi_n)\dots))\,d\mu,$$
which is the desired conclusion \eqref{eq-3}. 

Recall that, by Theorem \ref{th4.1}, we have $R=V_1^*V_2$ , and so $R^*=V_2^*V_1$. 
\end{proof}
%
%
%

The next result is an extension of Lemma \ref{lemn.1}(ii). Note that \eqref{eqn.4} is the assertion that $V_1^*$ intertwines the two operations, $R$ and $f\mapsto f\circ\sigma$. The next result shows that, by contrast, $V_2^*$ acts as a multiplier.

\begin{corollary}\label{cor2-}
Let $B,R,\Sigma$ and $\mu=\Sigma\circ\pi_1^{-1}$ be as in Proposition \ref{pr-2}, and set $\rho:=R^*1\in L^2(B,\mu)$; then 
$$(V_2^*(f\circ\sigma))(x)=\rho(x)\be_x(f)=\rho(x)(V_1^*f)(x),\quad(x\in B, f\in L^2(\Omega,\Sigma)).$$
\end{corollary}

\begin{proof}
This is immediate from Proposition \ref{pr-2}, see \eqref{eq-3}. Recall that the span of the tensors is dense in $L^2(\Omega,\Sigma)$ and that if $f=(\varphi_1\circ\pi_1)(\varphi_2\circ\pi_2)\dots(\varphi_n\circ\pi_n)$, then $f\circ\sigma=(\varphi_1\circ\pi_2)(\varphi_2\circ\pi_3)\dots(\varphi_n\circ\pi_{n+1})$. 

\end{proof}

In Proposition \ref{pr4.6} we calculate the multiplier $\rho$ for the special case of the wavelet representation from Example \ref{ex1.4.1}.

\begin{corollary}\label{cor5.3}
Let $B$ and $R$ be as in Theorem \ref{th4.1} , and let $\mu\in \M_1(B)$ be given. The induced measure on $\Omega=B^\bn$ is denoted $\Sigma^{(\mu)}$ and specified as in \eqref{eq1.7.1}.  We then have the following equivalence:
\begin{enumerate}
	\item $h\in L^2(B,\mu)$ and $Rh=h$, i.e., $h$ is harmonic; and 
	\item There exists $f\in L^2(\Omega,\Sigma^{(\mu)})$ such that $f=f\circ\sigma$ and 
	\begin{equation}
h(x)=\int_{\pi_1^{-1}(x)} f\,d\bp_x^{(R)}.
\label{eq5.3.1}
\end{equation}
\end{enumerate}
\end{corollary}

\begin{proof}
The implication (ii)$\Rightarrow$(i) follows from Lemma \ref{lemn.1} and Remark \ref{remn.2}. For (i)$\Rightarrow$(ii), let $h$ be given and assume it satisfies (i). An application of (iii) from Theorem \ref{th4.1} now yields 
$$Q_1(h\circ \pi_{n+1})=R^n(h)\circ\pi_1=h\circ\pi_1=V_1h.$$

Using \eqref{eq4.1.2}, we get $V_1(h-V_1^*(h\circ\pi_{n+1}))=0$ for all $n=0,1,2,\dots$; and therefore 
\begin{equation}
h=V_1^*(h\circ\pi_{n+1}),\quad(n\in\bn).
\label{eq5.3.2}
\end{equation}
Recalling
\begin{equation}
V_1^*(h\circ \pi_{n+1})(x)=\int_{\pi_1^{-1}(x)}h\circ\pi_{n+1}\,d\bp_x^{(R)}
\label{eq5.3.3}
\end{equation}
and using Theorem \ref{th1.7}, we conclude that $\{h\circ\pi_{n+1}\}_{n\in\bn}$ is a bounded $L^2(\Omega,\Sigma^{(\mu)})$-martingale. 

By Doob's theorem, there is a $f\in L^2(\Omega,\Sigma^{(\mu)})$ such that 
$$\lim_{n\rightarrow\infty}\|f-h\circ\pi_{n+1}\|_{L^2(\Sigma^{(\mu)})}=0.$$
Since $\pi_{n+1}\circ\sigma=\pi_n$, it follows that $f\circ\sigma=f$. 
Taking the limit in \eqref{eq5.3.2} and using that the operator norm of $V_1^*$ is one, we get that $h=V_1^*f$ and therefore the desired formula \eqref{eq5.3.1} holds. 
\end{proof}

\subsection{A stochastic process indexed by $\bn$}
\begin{remark}
In the literature one has a number of theorems dealing with the existence of measures $\mu$ satisfying the various conditions; and if $\mu\circ R=\mu$ is satsified, then the measure is called a Ruelle equilibrium measure. 
\end{remark}
\begin{theorem}\label{th2.1}
Let $B$ be compact Hausdorff and $R:C(B)\rightarrow C(B)$ positive, $R1=1$. Let $\mu\in \M_1(B)$ such that $\mu(B)=1$, $\mu\circ R=\mu$. Set 
$$X_n(\varphi)=\varphi\circ\pi_n,\quad(\varphi\in C(B), n\in\bn)$$
and
\begin{equation}
\int f\,d\Sigma=\int_B\int _{\pi_1^{-1}(x)} f\,d\bp_x^{(R)}\,d\mu(x)
\label{eq2.1}
\end{equation}
Then 
$$\be(\dots)=\int\dots\,d\Sigma$$
satisfies 
\begin{equation}
\be(X_n(\varphi)X_{n+k}(\psi))=\int_B\varphi(x)(R^k\psi)(x)\,d\mu,\quad(n,k\in\bn,\varphi,\psi\in C(B))
\label{eq2.2}
\end{equation}
i.e., $R^k$ is the transfer operator governing distances $k$. Asymptotic properties as $k$ goes to infinity govern long-range order. 
\end{theorem}

\begin{proof}
From the definition of $\bp_x^{(R)}$ we have 
\begin{equation}
\bp_x^{(R)}(X_n(\varphi))=R^{n-1}(\varphi)(x),\quad \varphi\in C(B)
\label{eq2.3*}
\end{equation}
Now let $n,k,\varphi,\psi$ as in the statement in \eqref{eq2.2}. Let $\Sigma$ be the measure on $B^{\bn}$ in \eqref{eq2.1}. Then
$$\be(X_n(\varphi)X_{n+k}(\psi))=\int_{B^{\bn}}(\varphi\circ\pi_n)(\psi\circ\pi_{n+k})\,d\Sigma=\int_B R^{n-1}(\varphi R^k(\psi))(x)\,d\mu(x)=\int\varphi(x) R^k(\psi)(x)\,d\mu(x)$$
which is the desired conclusion.

\end{proof}

\begin{definition}

We say that $\{X_k(\varphi)\}$ is independent at $\infty$ if 
\begin{equation}
\lim_{k\rightarrow\infty}\be(X_n(\varphi)X_{n+k}(\psi))=\left(\int\varphi\,d\mu\right)\left(\int\psi\,d\mu\right),\quad(\varphi,\psi\in C(B),n\in\bn).
\label{eq2.5}
\end{equation}
\end{definition}

\begin{corollary}
Suppose for all $\varphi$ in $C(B)$ we have 
$$\lim_{k\rightarrow\infty} R^k(\varphi)=\left(\int\varphi\,d\mu\right)1,$$
then \eqref{eq2.5} is satisfied.

\end{corollary}

\begin{proof}
We proved

$$\be(X_n(\varphi)X_{n+k}(\psi))=\int \varphi R^k\psi\,d\mu.$$

Now take the limit as $k\rightarrow\infty$, the desired conclusion \eqref{eq2.5} follows. 
\end{proof}

The next result answers the question: what is the distribution of the random variable $X_n(\varphi)$?
\begin{corollary} Assume $\mu\circ R=\mu$. 
The distribution of $X_n(\varphi)$ is $\mu(\{x\in B: \varphi(x)\leq t\})$ for all $n$. 

\end{corollary}

\begin{proof}
Take $\varphi$ real valued for simplicity. 

For $t\in\br$,
$$\Sigma(\{\tilde x\in B^\bn: \varphi\circ\pi_n(\tilde x)\leq t\})=\int_B\int_{\pi_1^{-1}(x)}\chi_{\{\varphi\leq t\}}\circ\pi_n\,d\bp_x^{(R)}\,d\mu(x)=\int_B R^{n-1}\chi_{\{\varphi\leq t\}}\,d\mu=\int_B\chi_{\{\varphi\leq t\}}\,d\mu.$$ In particular, it follows that all the random variables $X_n(\varphi)$ have the same distribution. 

\end{proof}

\subsection{Application to random walks}

\begin{corollary}\label{cor1.8}
Let $(r,W)$ be as in Definition \ref{def1.1}, and let $R_W$ be the Ruelle operator in \eqref{eq1.3}, $\bp_x^{(W)}$ - the random walk measure with transition probability specified as follows
\begin{equation}
\prob(x\rightarrow y)=\left\{\begin{array}{cc}
W(y),&\mbox{ if }r(y)=x\\
0,&\mbox{otherwise}\end{array}\right. 
\label{eq1.16}
\end{equation}
Then $\bp_x$ from Theorem \ref{th1.7} is equal to $\bp_x^{(W)}$.
\end{corollary}

\begin{proof}
We apply Theorem \ref{th1.7} to $R=R_W$ in \eqref{eq1.3} and we compute the right-hand side in \eqref{eq1.13} with induction
$$(M_{\varphi_1}R_W\dots R_W M_{\varphi_{n+1}})(x)=\varphi_1(x)\sum_{y_1}\dots\sum_{y_n}W(y_1)W(y_2)\dots W(y_n)\varphi_2(y_1)\dots\varphi_{n+1}(y_n),$$
where $r(y_{i+1})=y_i$, $1\leq i<n$, $r^{(n)}(y_n)=x$. Further,
$$=\sum\dots\sum \prob(x\rightarrow y_1)\prob(y_1\rightarrow y_2\, |, y_1)\dots\prob(y_{n-1}\rightarrow y_n\,|\,y_{n-1})\varphi(y_1)\dots\varphi(y_n)$$
\begin{equation}
=\int d\bp_x^{(\mbox{$W$-transition $R_W$-measure})}\varphi_1\otimes\dots\otimes\varphi_n
\label{eq1.17}
\end{equation}

\end{proof}

\begin{remark}\label{rem1.9}
The assertion in \eqref{eq1.17} applies to any random walk measure, for example, the one in Example \ref{ex1.4.2}.

Let $G=(V,E)$ be as in Example \ref{ex1.4.2}, with $E$ un-directed edges. Let $c:E\rightarrow[0,\infty)$ be such that 
\begin{equation}
c_{(xy)}=c_{(yx)}\mbox{ for all }(xy)\in E,\quad c_{(xy)}\neq 0\mbox{ if }(xy)\not\in E.
\label{eq1.18}
\end{equation}
A function as in \eqref{eq1.18} is called conductance. 

Set $p=p^c$, where 
\begin{equation}
p_{xy}=\frac{c_{xy}}{\sum_{z, z\sim x}c_{xz}}=\frac{c_{xy}}{c(x)},
\label{eq1.19}
\end{equation}
where 
$$c(x)=\sum_{z, z\sim x}c_{xz},\mbox{ and } z\sim x\mbox{ means }(zx)\in E.$$
Then there is a unique $\bp_x^{(c)}$ such that 
$$\int \varphi_1\otimes\dots\otimes\varphi_n\,d\bp_x^{(c)}=\sum_{y_1}\dots\sum_{y_n}p_{xy_1}p_{y_1y_2}\dots p_{y_{n-1}y_n}\varphi_1(y_1)\dots\varphi_n(y_n),$$
where the sums are over all $y_1,y_2,\dots,y_n$ such that $(y_iy_{i+1})\in E$.

Note that $\bp_x^{(W)}$ is supported on the solenoid, and $\bp_x^{(c)}$ is supported on $S^{(G)}$ (see \eqref{eq1.10}).

\end{remark}

\begin{remark}\label{rem1.10}
The last application is useful in the setting of harmonic functions on graphs $G=(V,E)$ with prescribed conductance function $c$ as in \eqref{eq1.18}. Set 
\begin{equation}
(\Delta\varphi)(x)=\sum_{y\in V,  y\sim x}c_{xy}(\varphi(x)-\varphi(y))
\label{eq1.21}
\end{equation}
the graph Laplacian with conductance $c$.

A function $\varphi$ on $V$ satisfies $\Delta\varphi\equiv 0$, iff 
\begin{equation}
\varphi(x)=\sum_{y\in V,y\sim x}p_{xy}^{(c)}\varphi(y),
\label{eq1.22}
\end{equation}
where $p_{xy}^{(c)}=\frac{c_{xy}}{c(x)}$ as in \eqref{eq1.19}.
\end{remark}

{\bf Application.} Use $\bp_x$ to get harmonic functions.
The study of classes of harmonic functions is of interest for infinite networks (see Remarks \ref{rem1.9} and \ref{rem1.10}), and in Corollary \ref{cor5.3} is is shown that the harmonic functions $h$ are precisely those that arise from applying $\be_x$ to functions $f$, $f\circ\sigma=f$ on $B^{\bn}$, i.e.,

$$h(x)=\int f(x\dots)\,d\bp_x,$$
and conversely a martingale limit constructs $f$ from $h$. For more details on this construction, see Corollary \ref{cor5.3}.

\subsection{An application to integral operators}

Let $K:B\times B\rightarrow[0,\infty)$ be a continuous function and let $\mu$ be a probability measure on $B$ such that 
\begin{equation}
\int_B K(x,y)\,d\mu(y)=1\mbox{ for all }x\in B.
\label{eq1.24}
\end{equation}
Define 
$$R_Kf(x)=\int_BK(x,y)f(y)\,d\mu(y),\quad(x\in B,f\in C(B)).$$
Then 
$R=R_K$ defines a positive operator as in Definition \ref{def1.2}, $R_K1=1$ and then $\bp_x$ in Theorem \ref{th1.7} satisfies 
$$\int_{B^{\bn}}\varphi_1\otimes\dots\otimes\varphi_{n+1}\,d\bp_x=\varphi_1(x)\int\dots\int K(x,y_1)K(y_1,y_2)\dots K(y_{n-1},y_n)\varphi_2(y_1)\dots\varphi_{n+1}(y_n)\,d\mu(y_1)\dots\,d\mu(y_n)$$
We get a measure $\Sigma$ on $B^{\bn}$ as follows 
\begin{equation}
\int f\,d\Sigma=\int f\,d\bp_x\,d\mu(x)
\label{eq1.25}
\end{equation}
since the right-hand side in \eqref{eq1.25} is independent of $x$.

\section{Positive operators and endomorphisms}

\subsection{Preliminaries about $r:B\rightarrow B$} Given an endomorphism $r$, we form the solenoid $\Sol(r)\subset B^{\bn}$. Below we will study $\widehat r:\Sol(r)\rightarrow \Sol(r)$, $$\widehat r(x_1x_2\dots)=(r(x_1)x_1x_2\dots)$$ and $\widehat r\in\operatorname*{Aut}(\Sol(r))$.

Given a positive operator $R:C(B)\rightarrow C(B)$, $R1=1$ we then form the measure $\bp_x^{(R)}$ in the usual way. We will prove the following property $\bp_x^{(R)}\circ \widehat r^{-1}=\bp_x^{(R)}$ on the solenoid but not on $B^{\bn}$.

We will impose the conditon \eqref{eq3.3} $$R((\varphi\circ r)\psi)=\varphi R\psi$$ as the only axiom. It may or may not be satisfied for some examples of positive operators $R$. But it does hold in the following two examples:

$$(R\varphi)(x)=\sum_{r(y)=x}W(y)\varphi(y)\mbox{ and }$$

$$(R\varphi)(x)=\frac{1}{\#r^{-1}(x)}\sum_{r(y)=x}|m(y)|^2\varphi(y),$$
where the functions $W$ and $m$ are given subject to the usual conditions.

For reference to earlier papers dealing with measures on infinite products, random walk, and stochastic processes; see e.g., \cite{JP11, JP10, AJ12}.

\begin{example}\label{ex2.29} Classical wavelet theory on the real line. 
Let $N=2$, $B=\bt=\{z\in\bc : |z|=1\}\simeq\br/\bz\simeq(-\frac12,\frac12]$ via $z=e^{2\pi i\theta}$, $\theta\in\br/\bz$; $\mu=d\theta$; $L^2(B,\mu)=L^2((-\frac12,\frac12],d\theta)$, $r:B\rightarrow B$, 
\begin{equation}
r(z)=z^2,\mbox{ or equivalently }r(\theta\mod\bz)=2\theta\mod\bz.
\label{eq2.29.1}
\end{equation}
Let
\begin{equation}
m_0(\theta)=\sum_{n\in\bz}h_ne^{2\pi i n\theta},\mbox{ or equivalently }m_0(z)=\sum_{n\in\bz}h_nz^n,
\label{eq2.29.2}
\end{equation}
where we assume 
$$\sum_{n\in\bz}h_n=\sqrt{2},\quad\sum_{n\in\bz}|h_n|^2<\infty.$$

\begin{lemma}\label{lem2.29.1}
With $m_0$ as in \eqref{eq2.29.2}, the condition \eqref{eq1.5.1} is equivalent to 
\begin{equation}
\sum_{k\in\bz}h_k\cj h_{k-2n}=\frac12\delta_{n,0}.
\label{eq2.29.3}
\end{equation}
\end{lemma}

\begin{proposition}\label{pr2.29.2}(\cite{BJ02,DuJo05}) Suppose that $m_0$ is as above, and that there is a solution $\varphi\in L^2(\br)$ satisfying 
\begin{equation}
\frac1{\sqrt2}\varphi\left(\frac x2\right)=\sum_{n\in\bz}h_n\varphi(x-n),
\label{eq2.29.4}
\end{equation}
and
\begin{equation}
\mbox{ The translates $\varphi(\cdot-n)$ are orthogonal in $L^2(\br)$, $n\in\bz$}.
\label{eq2.29.4.1}
\end{equation}
Set $W:L^2(\bt)\rightarrow L^2(\br)$, 
\begin{equation}
(W\xi)(x)=\sum_{n\in\bz}\widehat\xi(n)\varphi(x-n)=:\pi(\xi)\varphi,
\label{eq2.29.5}
\end{equation}
where $\xi\in L^2(\bt)$, and $\widehat\xi(n)=\int_{\bt}\cj e_n\xi\,d\mu$;
\begin{equation}
(S_0\xi)(z)=m_0(z)\xi(z^2),\quad(z\in\bt);
\label{eq2.29.6}
\end{equation}
and
\begin{equation}
(Uf)(x)=\frac1{\sqrt2}f\left(\frac x2\right),\quad f\in L^2(\br).
\label{eq2.29.7}
\end{equation}
\begin{enumerate}
	\item Then $S_0$ is isometric, and $(L^2(\br),\varphi,\pi,U)$ is a wavelet representation. 
	\item The dilation $W:L^2(\bt)\rightarrow L^2(\br)$ then takes the following form: $W$ is isometric and it intertwines $S_0$ and the unitary operator $U$, i.e., we have
	\begin{equation}
(WS_0\xi)(x)=(UW\xi)(x)=\frac1{\sqrt2}(W\xi)\left(\frac x2\right),\quad (\xi\in L^2(\bt), x\in\br).
\label{eq2.29.8}
\end{equation}
\end{enumerate}

\end{proposition}

\begin{remark}\label{rem2.29.3}
With $m_0$ as specified in Proposition \ref{pr2.29.2}, we conclude that the wavelet representation can be realized on $L^2(\br)$. On the other hand, we will see in Corollary \ref{cor4.5} that it can be also realized on the solenoid. The two representations have to be isomorphic. The identifications can be done via the usual embedding of $\br$ into the solenoid $x\mapsto (e^{2\pi i x},e^{2\pi ix/2},e^{2\pi i x/2^2},\dots)$. The measure $\Sigma$ in this case is supported on the image of $\br$ under this embedding. For details, see \cite{Dut06}.
\end{remark}
\end{example}

{\bf Axioms.} $B$ compact Hausdorff space, $R:C(B)\rightarrow C(B)$ positive linear operator such that $R1=1$, $r:B\rightarrow B$ onto, continuous. 

Assume 
\begin{equation}
R((\varphi\circ r)\psi)=\varphi R(\psi),\quad (\varphi,\psi\in C(B))
\label{eq3.3}
\end{equation}

Note that \eqref{eq3.3} is the only property that we assume on the operator $R$. 

\begin{lemma}\label{lem3.0}
On the solenoid $$\operatorname*{Sol}(r)=\{(x_1,x_2,\dots)\in B^\bn : r(x_{i+1})=x_i\},$$
$$\pi_i\circ\widehat r=r\circ\pi_i.$$
\end{lemma}

\begin{proof}
For $\tilde x=(x_1,x_2,\dots)$, 
\begin{equation}
\widehat r(\tilde x)=(r(x_1),x_1,x_2,\dots).
\label{eq3.0.1}
\end{equation}
$$\pi_i\circ\widehat r(\tilde x)=x_{i-1}=r(x_i)=r\circ\pi_i(\tilde x).$$
\end{proof}

\begin{remark}\label{rem3.1}
Our initial setup for a given endomorphism $r$ in our present setup is deliberately left open to a variety of possibilities. Indeed, the literature on solenoid analysis is vast, but divides naturally into cases when  $r: B\rightarrow B$ has only one contractivity degree; as opposed to a mix of non-linear contractive directions. The first case is common in wavelet analysis, such as those studied in \cite{DuJo06w, DuJo07, DJ10, DJ12}. Examples of the second class, often called `` hyperbolic'' systems, includes the Smale-Williams attractor, with the endomorphism  $r$ there prescribed to preserve a foliation by meridional disks; see e.g., \cite{Ku10, KP07, KP07, Ru04}. Or the study of complex dynamics and Julia sets; see e.g., \cite{BCMN04} .
\end{remark}

\begin{lemma}\label{lem3.0.0}
Let $r:B\rightarrow B$ be given and let $\widehat r\in\operatorname*{Aut}(\Sol(r))$ be the induced automorphism on the solenoid. Then 
$$\widehat r(\pi_1^{-1}(x))=\pi_1^{-1}(r(x))\cap \pi_2^{-1}(x),\quad (x\in B).$$
\end{lemma}

\begin{proof}
Use the definition of $\widehat r$ in \eqref{eq3.0.1}. 
\end{proof}
\begin{definition}\label{def3.1}
Given $\mu$ and $R$, they generate the probability measure $\Sigma=\Sigma^{(\mu)}$ on $B^\bn$. We assume $R1=1$ and $\mu(B)=1$. Define
\begin{equation}
\be(f)=\int_{B^\bn}f\,d\Sigma
\label{eq3.1.1}
\end{equation}
\begin{equation}
\be_x(f):=\be(f\, |\, \pi_1=x)=\int_{\pi_1^{-1}(x)}f\,d\Sigma 
\label{eq3.1.2}
\end{equation}
\begin{equation}
\be_{x_1,x_2}(f):=\be(f \,|\, \pi_1=x_1,\pi_2=x_2)=\int_{\pi_1^{-1}(x_1)\cap \pi_2^{-1}(x_2)}f\,d\Sigma.
\label{eq3.1.3}
\end{equation}
for all $x_1,x_2\in B$. 
As before we take 
\begin{equation}
\be(f)=\int_B\int_{\pi^{-1}(x)}f\,d\bp_x\,d\mu(x)
\label{eq3.1.4}
\end{equation}
and we then get
\begin{equation}
\be_x(f)=\int_{\pi^{-1}(x)}f\,d\bp_x
\label{eq3.1.5}
\end{equation}

\end{definition}

\begin{lemma}\label{lem2.1}
Let $B$, $R$ and $r$ be given as above. 
\begin{equation}
R((\varphi\circ r)\psi)=\varphi R(\psi)
\label{eq2.3}
\end{equation}
Then the following two are equivalent for some measure $\mu$ on $B$:
\begin{enumerate}
	\item $\mu\circ R=\mu$
	\item $\int(\varphi\circ r)\psi\,d\mu=\int \varphi R\psi\,d\mu$. 
\end{enumerate}
\end{lemma}

\begin{proof}
(i)$\Rightarrow$(ii). Assume (i) and \eqref{eq2.3}. Then 
$$\int \varphi\circ r\cdot \psi\,d\mu=\int R((\varphi \circ r)\psi)\,d\mu=\int \varphi R\psi\,d\mu$$
which is condition (ii).

(ii)$\Rightarrow$(i). Assume (ii). Then set $\varphi=1$ in (ii) and we get $\int \psi\,d\mu=\int R\psi\,d\mu$ which is the desired property (i).
\end{proof}

\begin{lemma}\label{lem3.1.2}
Assume the basic axiom \eqref{eq2.3}. For $f\in L^1(\Sigma)$, we denote by $\be_\bullet(f)$, the function $x\mapsto \be_x(f)$, $x\in B$. Then
\begin{equation}
\be_x(f\circ\sigma)=R(\be_\bullet(f))(x)
\label{eq3.1.6}
\end{equation}
Also,
\begin{equation}
\be_x(f\circ\widehat r)=\be_{r(x),x}(f)
\label{eq3.1.7}
\end{equation}

for all $x\in B$, $f\in L^1(\Sigma)$, or equivalently
\begin{equation}
\be(f\circ\widehat r\,|\, \pi_1=x)=\be(f\,|\,\pi_1=r(x),\pi_2=x);
\label{eq3.1.8}
\end{equation}
see the notations in Definition \ref{def3.1}.
\end{lemma}

\begin{proof}
Equation \eqref{eq3.1.6} is proved in \eqref{eq2.3.7}. For \eqref{eq3.1.7}, we use the Stone-Weierstrass approximation as before. If 
$$f=(\varphi_1\circ\pi_1)(\varphi_2\circ\pi_2)\dots(\varphi_n\circ\pi_n),$$
then 

$$f\circ\widehat r=(\varphi_1\circ r\circ\pi_1)(\varphi_2\circ r\circ\pi_2)\dots(\varphi_n\circ r\circ\pi_n),$$
and so 
$$\be_x(f\circ\widehat r)=\varphi_1(r(x))R(\varphi_2\circ rR(\varphi_3\circ r\dots R(\varphi_n\circ r))\dots)(x)$$
$$=\varphi_1(r(x))\varphi_2(x)R(\varphi_3R(\dots\varphi_{n-1}R(\varphi_n))\dots)(x)=\be_{r(x),x}(f),$$
or equivalentlly, \eqref{eq3.1.8}.
\end{proof}

\begin{proposition}\label{pr3.3}
Let $B$ and $R:C(B)\rightarrow C(B)$ be as stated in Theorem \ref{th1.7}. For every $\mu\in \M_1(B)$ we denote the induced measure on $\Omega=B^{\bn}$ by $\Sigma^{(\mu)}$. If some $r:B\rightarrow B$ satisfies 
\begin{equation}
R((\varphi\circ r)\psi )=\varphi R(\psi),\quad(\varphi,\psi\in C(B))
\label{eq3*.*}
\end{equation}
then every one of the induced measures $\Sigma^{(\mu)}$ has its support contained in the solenoid $\operatorname*{Sol}(r)$. 
\end{proposition}

\begin{proof}
Using Lemma \ref{lemn.1}, it is enough to prove that each of the measures $\bp_x^{(R)}$ with $x$ fixed (from Corollary \ref{cor1.10}) has its support equal to 
\begin{equation}
\pi_1^{-1}(x)\cap\operatorname*{Sol}(r)
\label{eq3*.2}
\end{equation}
For every $n$, consider all infinite words indexed by $y\in r^{-n}(x)$ and specified on the beginning length-$n$ segments as follows
$\Omega_n(r,x): (x,r^{n-1}(y),\dots,r(y),y,\mbox{ free infinite tail})$ and note that 
\begin{equation}
\pi_1^{-1}(x)\cap\operatorname*{Sol}(r)=\bigcap_n\Omega_n(r,x)
\label{eq3.12}
\end{equation}

For $n=1$, we have 
$$\bp_x^{(R)}(\Omega_1(r,x))=R(\chi_{\{x\}}\circ r)(x)=\chi_{\{x\}}(x)R(1)=1,$$
where we used assumption \eqref{eq3*.*} in the last step in the computation. 

The remaining reasoning in the proof is an induction. Indeed, one checks that
$$\bp_x^{(R)}(\Omega_n(r,x))=R((\chi_{\{x\}}\circ r)R((\chi_{\{x\}}\circ r^2)R(\dots(\chi_{\{x\}}\circ r^{n-1})R(\chi_{\{x\}}\circ r^n)\dots)))(x)$$ 
$$=R((\chi_{\{x\}}\circ r)R(\dots R(\chi_{\{x\}}\circ r^{n-1})\dots))(x).$$
Hence the assertion for $n-1$ implies the next step $n$. By induction, we get
$$\bp_x^{(R)}(\Omega_n(r,x))=1,\quad(n\in\bn,x\in B).$$

Using \eqref{eq3.12}, we get
$$\bp_x^{(R)}(\pi_1^{-1}(x)\cap \operatorname*{Sol}(r))=\lim_{n\rightarrow\infty}\bp_x^{(R)}(\Omega_n(r,x))=1.$$
As a consequence, the measure $\bp_x^{(R)}$ assigns value 1 to the indicator function of $\pi_1^{-1}(x)\cap\operatorname*{Sol}(r)$. But 
$$\operatorname*{Sol}(r)=\bigcup_{x\in B}\pi_1^{-1}(x)\cap \operatorname*{Sol}(r).$$
So if $\mu(B)=1$, it follows from \eqref{eq2.1} that 
$$\Sigma^{(\mu)}(\operatorname*{Sol}(r))=\int_{B^\bn}\chi_{\operatorname*{Sol}(r)}\,d\Sigma^{(\mu)}=1;$$
and as a result that
$$\Sigma^{(\mu)}(B^\bn\setminus\operatorname*{Sol})=0$$
which is the desired conclusion.

\end{proof}

\begin{corollary}\label{cor3.2}
Let $B,r,\mu,R$ be as above and assume \eqref{eq2.3}.
 Then $\Sigma$ is supported on $\Sol(r)$ and $\widehat r$ is invertible on $\Sol(r)$ with $\widehat r^{-1}=\sigma$.
	 The measure $\Sigma$ is invariant (for $\widehat r$) if and only if
	\begin{equation}
	\mu\circ R=\mu.
\label{eq3.1.9}
\end{equation}

\end{corollary}

\begin{proof}
It is enough to prove that 
\begin{equation}
\int_{\Sol(r)}f\circ \sigma\,d\Sigma=\int_{\Sol(r)}f\,d\Sigma
\label{eq3.1.10}
\end{equation}
holds for all $f\in L^1(\Sigma)$ if and only if \eqref{eq3.1.9} holds. But, by \eqref{eq3.1.6} we have 
$$\int_{\Sol(r)}f\circ\sigma\,d\Sigma=\int_B\be_x(f\circ\sigma)\,d\mu(x)=\int_BR(\be_\bullet(f))(x)\,d\mu(x),$$
and
$$\int_{\Sol(r)}f\,d\Sigma=\int_B\be_x(f)\,d\mu(x).$$
But the functions $x\mapsto \be_x(f)$ are dense in $L^1(B,\mu)$ as $f$ varies in $L^1(\Sigma)$ (consider for example $f=g\circ\pi_1$ for $g\in C(B)$). Thus the equivalence of \eqref{eq3.1.9} and \eqref{eq3.1.10} is immediate from this.
\end{proof}

\begin{corollary}\label{cor3.3} Let $R$ be a positive operator in $C(B)$ satsifying the axioms above, $R1=1$, $R((\varphi\circ r)\psi)=\varphi R(\psi)$ for all $\varphi,\psi\in C(B)$. Let $\mu$ be a Borel measure on $B$, and set $\Sigma=\Sigma^{(\mu)}$
$$\int_{\Sol(r)}f\,d\Sigma:=\int_B\int_{\pi_1^{-1}(x)} f\,d\bp_x^{(R)}\,d\mu(x);$$
then 
\begin{equation}
\U^{(R)}f:=f\circ \widehat r
\label{eq3.10}
\end{equation}
defines a unitary operator on $L^2(\Sol(r),\Sigma)$ if and only if $\mu=\mu\circ R$. 
\end{corollary}

\begin{proof}
Since $\widehat r$ is invertible in $\Sol(r)$, we conclude that $\U^{(R)}$ maps onto $L^2(\Sol(r),\Sigma)$. Recall $\widehat r^{-1}=\sigma$
\begin{equation}
\sigma(x_1,x_2,x_3,\dots)=(x_2,x_3,x_4,\dots)
\label{eq3.11}
\end{equation}
Since, by Proposition \ref{pr3.3}, the measure $\Sigma$ is supported on $\Sol(r)$, the result follows from Corollary \ref{cor3.2}.

\end{proof}

\begin{corollary}\label{cor1.11}
Let $B$ and $R:C(B)\rightarrow C(B)$ be as in Corollary \ref{cor1.10}. Let $\mu\in \M_1(B)$ and consider $\Sigma=\Sigma^{(\mu)}$. Let $r:B\rightarrow B$ be an endomorphism.
\begin{enumerate}
	\item For the operators $V_1:L^2(B,\mu)\rightarrow L^2(\Sol(r),\Sigma)$ and $\U^{(R)}: f\mapsto f\circ\widehat r$ acting in $L^2(\Sol(r),\Sigma)$ we have the following covariance relation: $V_1$ is isometric and 
	$$(V_1^*\U^{(R)}V_1)(\varphi)=\varphi\circ r,\quad (\varphi\in C(B)).$$
	\item Assume that $\mu=\mu\circ R$ and 
	\begin{equation}
R((\varphi\circ r)\psi)=\varphi R(\psi),\quad(\varphi,\psi\in C(B))
\label{eq1.11.1}
\end{equation}
holds.  For functions $F$ on $\Omega$, say $F\in L^\infty(\Sol(r))$, let $M_F$ be the multiplication operator defined by $F$. Then $\U^{(R)}$ is unitary and 
 the following covariance relation holds:
$$(\U^{(R)})^*M_F\U^{(R)}=M_{F\circ \sigma}.$$
\end{enumerate}

\end{corollary}

\begin{proof}
(i) We will make use of the formula \eqref{eq1.7.4} for $V_1^*$, see Corollary \ref{cor1.10}(ii). We now compute
$$(V_1^*\U^{(R)}V_1\varphi)(x)=\be_x\U^{(R)}V_1\varphi=\be_x(\varphi\circ\pi_1\circ\widehat r)=\be_x(\varphi\circ r\circ\pi_1)=(\varphi\circ r)(x),$$
which is the conclusion in (i). 

(ii) We proved in Corollary \ref{cor3.3} that $\U^{(R)}$ is unitary. The covariance relation follows from a simple computation.

\end{proof}

   For reference to earlier papers dealing with measures on infinite products, and shift-invariant systems; see e.g., \cite{CGH12, CH94}.

\subsection{Compact groups} As a special case of our construction, we mention the compact groups; this will include the case of wavelet theory.
\begin{proposition}\label{pr3-1}
Assume $B$ is a compact group with normalized Haar measure $\mu$. Let $r:B\rightarrow B$ be a homomorphism $r(xy)=r(x)r(y)$ for all $x,y\in B$, and assume for $N\in\bn$, $N>1$
$$\#r^{-1}(x)=N,\quad(x\in B).$$
Set 
\begin{equation}
(R\varphi)(x)=\frac1N\sum_{r(y)=x}\varphi(y),\quad(\varphi\in C(B));
\label{eq3-1.1}
\end{equation}
then
\begin{enumerate}
	\item  $\Sol(r)$ is a compact subgroup of $B^\bn$.
	\item The induced measure $\Sigma=\Sigma^{(\mu,R)}$ is the Haar measure on the group $\Sol(r)$. 
\end{enumerate}
\end{proposition}   

\begin{proof}
(i) If $\tilde x=(x_1,x_2,\dots), \tilde y=(y_1,y_2,\dots)\in \Sol(r)$ then $r(x_{i+1}y_{i+1})=r(x_{i+1})r(y_{i+1})=x_iy_i$, so $\tilde x\tilde y\in\Sol(r)$.

(ii) From \eqref{eq3-1.1} we see that the measures $\bp_x$ in the decomposition 
\begin{equation}
d\Sigma=\int_B\bp_x\,d\mu(x)
\label{eq3-1.2}
\end{equation}
$\bp_x\in\M(\pi_1^{-1}(x))$, $x\in B$, are random-walk measures with uniform distributions on the points in $r^{-1}(x)$, for all $x\in B$. If $\be$ denotes the $\Sigma$-expectation and $\be_x$ the $\bp_x$-expectation, we have
$$\be_x(f)=\be(f\,|\,\pi_1=x),\quad(x\in B).$$

For points $\tilde y=(y_1,y_2,\dots)\in\Sol(r)$, denote by $f(\cdot\tilde y)$ the translated function on $\Sol(r)$. Then
\begin{equation}
\be(f(\cdot \tilde y)\,|\,\pi_1=x)=\be(f\,|\,\pi_1=xy_1)
\label{eq3-1.3}
\end{equation}
where we use the terminology in Proposition \ref{pr2.3.2}.

Now, combining \eqref{eq3-1.2} and \eqref{eq3-1.3}, we arrive at the formula:
$$\int_{\Sol(r)}f(\cdot\tilde y)\,d\Sigma=\be(f(\cdot\tilde y))=\int_B\be(f(\cdot\tilde y)\,|\,\pi_1=x)\,d\mu(x)=\int_B\be(f\,|\,\pi_1=xy_1)\,d\mu(x)$$
$$=\int_B\be(f\,|\,\pi_1=x)\,d\mu(x)=\int_{\Sol(r)}f\,d\Sigma.$$
\end{proof}

\begin{remark}\label{rem3-2}
In wavelet theory, one often takes $B=\br^n/\bz^n$, and a fixed $n\times n$ matrix $A$ over $\bz$ such that the eigenvalues $\lambda$ satisfy $|\lambda|>1$. For $r:B\rightarrow B$, then take 
$$r(x\mod\bz^n)=Ax\mod\bz^n,\quad(x\in\br^n)$$
and it is immediate that $r$ satisfies the multiplicative property in Proposition \ref{pr3-1}.

There are important examples when $r:B\rightarrow B$ does not satisfy this property. 

\end{remark}

\begin{example}\label{ex3-3}(Non-group case: the Smale-Williams attractor) Take $B=\bt\times\bd$, where $\bt=\br/\bz$ and $\bd=\{z\in \bc : |z|\leq1\}$ the disk. For $(t,z)\in\bt\times\bd$, set 
$$r(t,z)=(2t\mod\bz,\frac14 z+\frac12 e^{2\pi i t}).$$
Then $\Sol(r)$ is the Smale-Williams attractor, see \cite{KP07}, a hyperbolic strange attractor.

\end{example}

\section{Isometries}
   
Below we study condition on functions $m:B\rightarrow \bc$ which gurantees that $L^2(\mu)\ni f\mapsto m\cdot f\circ r\in L^2(\mu)$ defines an isometry in $L^2(B,\mu)$; and we will study the unitary dilations 
$L^2(\Sol(r),\Sigma)\ni\tilde f\mapsto \tilde f\circ\widehat r\in L^2(\Sol(r),\Sigma)$.

{\bf Setting.} $B$ fixed compact Hausdorff space. We introduce 
\begin{enumerate}
	\item $r:B\rightarrow B$ measurable, onto such that 
	\begin{equation}
1\leq \#r^{-1}(x)<\infty,\quad(x\in B)
\label{eq4.1}
\end{equation}

\item
$\mu$ Borel measure on $B$, $\mu(B)=1$. 
\item $m:B\rightarrow \bc$ a fixed function on $B$. 
\end{enumerate}

Question: Given two of them what are the conditions that the third should satisfy such that  
\begin{equation}
L^2(\mu)\ni f\mapsto m\cdot f\circ r\in L^2(\mu)
\label{eq4.2}
\end{equation}
is an isometry.

\begin{definition} 
Transformations of measures. Given $\nu$ measure on $B$, $\nu\in \M_1(B)$ and $r:B\rightarrow B$, 
set $\nu\circ r^{-1}\in \M_1(B)$. For $A\in\B(B)$ a Borel set, $(\nu\circ r^{-1})(A):=\nu(r^{-1}(A))$ where $r^{-1}(A):=\{x\in B : r(x)\in A\}$.

Fact: $\nu\circ r^{-1}$ is determined uniquely by the condition 
\begin{equation}
\int_B\varphi\circ r\,d\nu=\int \varphi \,d(\nu\circ r^{-1}),\quad(\varphi\in C(B))
\label{eq4.5}
\end{equation}
\end{definition}

\begin{lemma}

Fix $r,\mu,m$; then \eqref{eq4.2} is satisfied iff 
\begin{equation}
(|m|^2\,d\mu)\circ r^{-1}=\mu
\label{eq4.6}
\end{equation}
\end{lemma}

\begin{definition}
Fix $r$, then we say that $\mu$ is {\it strongly invariant} iff 
\begin{equation}
\int \varphi(x)\,d\mu(x)=\int\frac{1}{\#r^{-1}(x)}\sum_{r(y)=x}\varphi(y)\,d\mu(x)
\label{eq4.7}
\end{equation}

\end{definition}

\begin{lemma}\label{lem4.4}
Given $r$ and assume $\mu$ is strongly invariant, then the isometry property \eqref{eq4.2} holds iff the corresponding positive operator 
$$(R\varphi)(x):=\frac{1}{\#r^{-1}(x)}\sum_{r(y)=x}|m(y)|^2\varphi(y)$$
satisfies $R1=1$. 

\end{lemma}

\begin{proof}
Substitute \eqref{eq4.7} into \eqref{eq4.6}. Note that then the equation 
$$\int_B\varphi(r(x))|m(x)|^2\,d\mu(x)=\int_B\varphi(x)\frac{1}{\#r^{-1}(x)}\sum_{r(y)=x}|m(y)|^2\,\mu(x),\quad (\varphi\in C(B))$$
holds; so $f\mapsto m\cdot f\circ r$ is isometric in $L^2(\mu)$ iff 
$$\frac{1}{\#r^{-1}(x)}\sum_{r(y)=x}|m(y)|^2=1$$
$\mu$-a.e. $x\in B$.

\end{proof}

The next corollary appears in \cite[Theorem 5.5]{DuJo07}.
\begin{corollary}\label{cor4.5}
Let $B$ and $r$ be as specified in section 3, let $\mu$ be strongly invariant and let the function $m$ be quadrature mirror filter, as in Example \ref{ex1.4.1}. Assume in addition that $m$ is non-singular, i.e. $$\mu(\{x: m(x)=0\})=0.$$ Then, with $R$ as in Lemma \ref{lem4.4}, we get a wavelet representation as in Theorem \ref{th1.6} to $L^2(\operatorname*{Sol}(r),\Sigma^{(\mu)})$ with 
\begin{enumerate}
	\item $\H=L^2(\Sol(r),\Sigma^{(\mu)})$;
	\item $\U f=(m\circ\pi_1)(f\circ\widehat r)$, for all $f\in L^2(\Sol(r),\Sigma^{(\mu)})$;
	\item $\pi(g)f=(g\circ\pi_1)f$, for all $g\in L^\infty(B)$, $f\in L^2(\Sol(r),\Sigma^{(\mu)})$;
	\item $\varphi=1$. 
\end{enumerate}
\end{corollary}

\begin{proof}
The details are contained in \cite[Theorem 5.5]{DuJo07} and require just some simple computations. We only have to check that our measure $\Sigma^{(\mu)}$ coincides with the one defined in \cite{DuJo07}. For this, we use \cite[Theorem 5.3]{DuJo07} and we have to check that 
$$\int \varphi\circ\pi_n\,d\Sigma^{(\mu)}=\int R^n(\varphi)\,d\mu,\quad (\varphi\in C(B)).$$
But this follows immediately from the definition of $\Sigma^{(\mu)}$ in \eqref{eq1.15}.
\end{proof}

\begin{proposition}\label{pr4.6}
Let $B,r,\mu$ and $m_0$ as in Example \ref{ex1.4.1}, i.e., $m_0$ is a QMF and the measure $\mu$ on $B$ is assumed strongly invariant with respect to $r$, and let $R$ as in Lemma \ref{lem4.4}. Then the function $\rho=R^*1$ in Corollary \ref{cor2-} is $\rho=|m_0|^2$. 
\end{proposition}

\begin{proof}
The result follows if we verify the formula for $R^*$; we have
\begin{equation}
(R^*\psi)(x)=|m_0(x)|^2\psi(r(x)).
\label{eq4-1}
\end{equation}
The derivation of \eqref{eq4-1} may be obtained as a consequence of strong invariance as follows: for all $\varphi,\psi\in C(B)$, we have:
$$\int_B|m_0|^2(\psi\circ r)\varphi\,d\mu=\int_B\psi(x)\frac1N\sum_{r(y)=x}|m_0(y)|^2\varphi(y)\,d\mu(x)=\int_B\psi(x)(R\varphi)(x)\,d\mu(x);$$
and the assertion \eqref{eq4-1} follows. 
\end{proof}

\begin{acknowledgements}
 One of the authors wishes to thank Professors Ka-Sing Lau, De-Jun Feng, and their colleagues, for organizing a wonderful conference in Hong-Kong, ``The International Conference on Advances of Fractals and Related Topics'', December 2012. Many discussions with participants at the conference inspired this paper. This work was partially supported by a grant from the Simons Foundation (\#228539 to Dorin Dutkay).
\end{acknowledgements}
\bibliographystyle{alpha}
\bibliography{eframes}

\def\cprime{$'$}
\begin{thebibliography}{DHSW11}

\bibitem[ABL11]{ABL11}
A.~B. Antonevich, V.~I. Bakhtin, and A.~V. Lebedev.
\newblock Crossed product of a {$C^*$}-algebra by an endomorphism, coefficient
  algebras, and transfer operators.
\newblock {\em Mat. Sb.}, 202(9):3--34, 2011.

\bibitem[ABL12]{ABL12}
A.~B. Antonevich, V.~I. Bakhtin, and A.~V. Lebedev.
\newblock A road to the spectral radius of transfer operators.
\newblock In {\em Dynamical systems and group actions}, volume 567 of {\em
  Contemp. Math.}, pages 17--51. Amer. Math. Soc., Providence, RI, 2012.

\bibitem[AJ12]{AJ12}
Daniel Alpay and Palle E.~T. Jorgensen.
\newblock Stochastic processes induced by singular operators.
\newblock {\em Numer. Funct. Anal. Optim.}, 33(7-9):708--735, 2012.

\bibitem[Arv86]{Ar86}
William Arveson.
\newblock Markov operators and {OS}-positive processes.
\newblock {\em J. Funct. Anal.}, 66(2):173--234, 1986.

\bibitem[Bal00]{Ba00}
Viviane Baladi.
\newblock {\em Positive transfer operators and decay of correlations},
  volume~16 of {\em Advanced Series in Nonlinear Dynamics}.
\newblock World Scientific Publishing Co. Inc., River Edge, NJ, 2000.

\bibitem[BCMN04]{BCMN04}
A.~F. Beardon, T.~K. Carne, D.~Minda, and T.~W. Ng.
\newblock Random iteration of analytic maps.
\newblock {\em Ergodic Theory Dynam. Systems}, 24(3):659--675, 2004.

\bibitem[BJ02]{BJ02}
Ola Bratteli and Palle Jorgensen.
\newblock {\em Wavelets through a looking glass}.
\newblock Applied and Numerical Harmonic Analysis. Birkh\"auser Boston Inc.,
  Boston, MA, 2002.
\newblock The world of the spectrum.

\bibitem[CGHU12]{CGH12}
J.-R. Chazottes, J.-M. Gambaudo, M.~Hochman, and E.~Ugalde.
\newblock On the finite-dimensional marginals of shift-invariant measures.
\newblock {\em Ergodic Theory Dynam. Systems}, 32(5):1485--1500, 2012.

\bibitem[CH94]{CH94}
M.~Courbage and D.~Hamdan.
\newblock Chapman-{K}olmogorov equation for non-{M}arkovian shift-invariant
  measures.
\newblock {\em Ann. Probab.}, 22(3):1662--1677, 1994.

\bibitem[Dau92]{Dau92}
Ingrid Daubechies.
\newblock {\em Ten lectures on wavelets}, volume~61 of {\em CBMS-NSF Regional
  Conference Series in Applied Mathematics}.
\newblock Society for Industrial and Applied Mathematics (SIAM), Philadelphia,
  PA, 1992.

\bibitem[DHSW11]{DHS11}
Dorin~Ervin Dutkay, Deguang Han, Qiyu Sun, and Eric Weber.
\newblock On the {B}eurling dimension of exponential frames.
\newblock {\em Adv. Math.}, 226(1):285--297, 2011.

\bibitem[DJ05]{DuJo05}
Dorin~Ervin Dutkay and Palle E.~T. Jorgensen.
\newblock Hilbert spaces of martingales supporting certain
  substitution-dynamical systems.
\newblock {\em Conform. Geom. Dyn.}, 9:24--45 (electronic), 2005.

\bibitem[DJ06]{DuJo06w}
Dorin~E. Dutkay and Palle E.~T. Jorgensen.
\newblock Wavelets on fractals.
\newblock {\em Rev. Mat. Iberoam.}, 22(1):131--180, 2006.

\bibitem[DJ07]{DuJo07}
Dorin~Ervin Dutkay and Palle E.~T. Jorgensen.
\newblock Martingales, endomorphisms, and covariant systems of operators in
  {H}ilbert space.
\newblock {\em J. Operator Theory}, 58(2):269--310, 2007.

\bibitem[DJ10]{DJ10}
Dorin~Ervin Dutkay and Palle E.~T. Jorgensen.
\newblock Spectral theory for discrete {L}aplacians.
\newblock {\em Complex Anal. Oper. Theory}, 4(1):1--38, 2010.

\bibitem[DJ11a]{DJ11a}
Dorin~Ervin Dutkay and Palle E.~T. Jorgensen.
\newblock Affine fractals as boundaries and their harmonic analysis.
\newblock {\em Proc. Amer. Math. Soc.}, 139(9):3291--3305, 2011.

\bibitem[DJ11b]{DJ11b}
Dorin~Ervin Dutkay and Palle E.~T. Jorgensen.
\newblock Spectral duality for unbounded operators.
\newblock {\em J. Operator Theory}, 65(2):325--353, 2011.

\bibitem[DJ12]{DJ12}
Dorin~Ervin Dutkay and Palle E.~T. Jorgensen.
\newblock Fourier duality for fractal measures with affine scales.
\newblock {\em Math. Comp.}, 81(280):2253--2273, 2012.

\bibitem[DJP09]{DJP09}
Dorin~Ervin Dutkay, Palle E.~T. Jorgensen, and Gabriel Picioroaga.
\newblock Unitary representations of wavelet groups and encoding of iterated
  function systems in solenoids.
\newblock {\em Ergodic Theory Dynam. Systems}, 29(6):1815--1852, 2009.

\bibitem[DJS12]{DJS12}
Dorin~Ervin Dutkay, Palle E.~T. Jorgensen, and Sergei Silvestrov.
\newblock Decomposition of wavelet representations and {M}artin boundaries.
\newblock {\em J. Funct. Anal.}, 262(3):1043--1061, 2012.

\bibitem[DL10]{DL10}
Xin-Han Dong and Ka-Sing Lau.
\newblock Cantor boundary behavior of analytic functions.
\newblock In {\em Recent developments in fractals and related fields}, Appl.
  Numer. Harmon. Anal., pages 283--294. Birkh\"auser Boston Inc., Boston, MA,
  2010.

\bibitem[DLS11]{DLS11}
Dorin~Ervin Dutkay, David~R. Larson, and Sergei Silvestrov.
\newblock Irreducible wavelet representations and ergodic automorphisms on
  solenoids.
\newblock {\em Oper. Matrices}, 5(2):201--219, 2011.

\bibitem[DMP08]{Dan08}
Jonas D'Andrea, Kathy~D. Merrill, and Judith Packer.
\newblock Fractal wavelets of {D}utkay-{J}orgensen type for the {S}ierpinski
  gasket space.
\newblock In {\em Frames and operator theory in analysis and signal
  processing}, volume 451 of {\em Contemp. Math.}, pages 69--88. Amer. Math.
  Soc., Providence, RI, 2008.

\bibitem[DS11]{DS11}
Dorin~Ervin Dutkay and Sergei Silvestrov.
\newblock Reducibility of the wavelet representation associated to the {C}antor
  set.
\newblock {\em Proc. Amer. Math. Soc.}, 139(10):3657--3664, 2011.

\bibitem[Dut06]{Dut06}
Dorin~Ervin Dutkay.
\newblock Low-pass filters and representations of the {B}aumslag {S}olitar
  group.
\newblock {\em Trans. Amer. Math. Soc.}, 358(12):5271--5291 (electronic), 2006.

\bibitem[FM12]{FM12}
Markus Fraczek and Dieter Mayer.
\newblock Symmetries of the transfer operator for {$\Gamma_0(N)$} and a
  character deformation of the {S}elberg zeta function for {$\Gamma_0(4)$}.
\newblock {\em Algebra Number Theory}, 6(3):587--610, 2012.

\bibitem[Hen12]{He12}
Doug Hensley.
\newblock Continued fractions, {C}antor sets, {H}ausdorff dimension, and
  transfer operators and their analytic extension.
\newblock {\em Discrete Contin. Dyn. Syst.}, 32(7):2417--2436, 2012.

\bibitem[JP10]{JP10}
Palle E.~T. Jorgensen and Erin Peter~James Pearse.
\newblock A {H}ilbert space approach to effective resistance metric.
\newblock {\em Complex Anal. Oper. Theory}, 4(4):975--1013, 2010.

\bibitem[JP11]{JP11}
Palle E.~T. Jorgensen and Erin P.~J. Pearse.
\newblock Gel\cprime fand triples and boundaries of infinite networks.
\newblock {\em New York J. Math.}, 17:745--781, 2011.

\bibitem[KP07]{KP07}
Sergey~P. Kuznetsov and Arkady Pikovsky.
\newblock Autonomous coupled oscillators with hyperbolic strange attractors.
\newblock {\em Phys. D}, 232(2):87--102, 2007.

\bibitem[Kuz10]{Ku10}
S.~P. Kuznetsov.
\newblock Example of blue sky catastrophe accompanied by a birth of
  {S}male-{W}illiams attractor.
\newblock {\em Regul. Chaotic Dyn.}, 15(2-3):348--353, 2010.

\bibitem[Kwa12]{Kw12}
B.~K. Kwa{\'s}niewski.
\newblock On transfer operators for {$C^*$}-dynamical systems.
\newblock {\em Rocky Mountain J. Math.}, 42(3):919--936, 2012.

\bibitem[LN12]{LN12}
Ka-Sing Lau and Sze-Man Ngai.
\newblock Martin boundary and exit space on the {S}ierpinski gasket.
\newblock {\em Sci. China Math.}, 55(3):475--494, 2012.

\bibitem[LR69]{LR69}
O.~E. Lanford, III and D.~Ruelle.
\newblock Observables at infinity and states with short range correlations in
  statistical mechanics.
\newblock {\em Comm. Math. Phys.}, 13:194--215, 1969.

\bibitem[LW09]{LW09}
Ka-Sing Lau and Xiang-Yang Wang.
\newblock Self-similar sets as hyperbolic boundaries.
\newblock {\em Indiana Univ. Math. J.}, 58(4):1777--1795, 2009.

\bibitem[MMS12]{MMS12}
Dieter Mayer, Tobias M{\"u}hlenbruch, and Fredrik Str{\"o}mberg.
\newblock The transfer operator for the {H}ecke triangle groups.
\newblock {\em Discrete Contin. Dyn. Syst.}, 32(7):2453--2484, 2012.

\bibitem[MP11]{MaPa11}
Matilde Marcolli and Anna~Maria Paolucci.
\newblock Cuntz-{K}rieger algebras and wavelets on fractals.
\newblock {\em Complex Anal. Oper. Theory}, 5(1):41--81, 2011.

\bibitem[Nau12]{Na12}
Fr{\'e}d{\'e}ric Naud.
\newblock The {R}uelle spectrum of generic transfer operators.
\newblock {\em Discrete Contin. Dyn. Syst.}, 32(7):2521--2531, 2012.

\bibitem[Rue96]{Ru96}
D.~Ruelle.
\newblock Sharp zeta functions for smooth interval maps.
\newblock In {\em International {C}onference on {D}ynamical {S}ystems
  ({M}ontevideo, 1995)}, volume 362 of {\em Pitman Res. Notes Math. Ser.},
  pages 188--206. Longman, Harlow, 1996.

\bibitem[Rue02]{Ru02}
David Ruelle.
\newblock Dynamical zeta functions and transfer operators.
\newblock {\em Notices Amer. Math. Soc.}, 49(8):887--895, 2002.

\bibitem[Rue04]{Ru04}
David Ruelle.
\newblock {\em Thermodynamic formalism}.
\newblock Cambridge Mathematical Library. Cambridge University Press,
  Cambridge, second edition, 2004.
\newblock The mathematical structures of equilibrium statistical mechanics.

\end{thebibliography}

\end{document}